\documentclass[a4paper,reqno]{amsart}
\usepackage[left=1in, right=1in, top=1in]{geometry}
\usepackage{amsmath,tikz,tkz-euclide,amsthm,tikz-cd}
\usepackage{float}
\usepackage{cancel,sansmath,bm}
\usepackage{amssymb,mathtools,booktabs,bbm}
\usepackage{enumitem,multicol}
\allowdisplaybreaks
\usetikzlibrary{trees, automata}
\usepackage{graphicx,url}

\tikzset{
	->,
	> = {Stealth[length=5pt]},
	node distance = 2cm,
	every state/.style = {rectangle, line width=0.35mm, minimum size=13pt, inner sep=0.65pt, font=\scriptsize},
	every node/.style = {font=\scriptsize},
	initial text = $ $
}

\renewcommand{\d}{\boldsymbol{\mathbf{d}}}
\renewcommand{\r}{\boldsymbol{\mathbf{r}}}
\newcommand{\N}{\mathbbm{N}}							
\newcommand{\Z}{\mathbbm{Z}}							
\newcommand{\Q}{\mathbbm{Q}}							
\newcommand{\R}{\mathbbm{R}}							
\newcommand{\D}{\mathcal{D}}
\newcommand{\A}{\mathcal{A}}

\newcommand{\sN}{\mathcal{N}}
\newcommand{\sR}{\mathcal{R}}

\newcommand{\brk}[1]{\left[ #1 \right]}
\newcommand{\brc}[1]{\left\{ #1 \right\}}
\newcommand{\abs}[1]{\left| #1 \right|}
\newcommand{\pren}[1]{\left( #1 \right)}

\newcommand{\flr}[1]{\left\lfloor #1 \right\rfloor}

\newcommand{\ol}[1]{\overline{#1}}
\newcommand{\mbf}[1]{\mathbf{#1}}

\let\oldlim\lim

\renewcommand{\lim}[3]{\oldlim\limits_{#1 \to #2}{#3}}

\setlist{leftmargin=*}											
\DeclareMathOperator{\tr}{tr}
\DeclareMathOperator{\Mod}{mod}

\newtheorem{theorem}{Theorem}
\newtheorem{lemma}[theorem]{Lemma}

\newtheorem{proposition}[theorem]{Proposition}

\theoremstyle{definition}
\newtheorem*{definition}{Definition}

\newcommand{\bmath}[1]{\boldsymbol{\mathbf{#1}}}
\newcommand{\e}{\mathbf{e}}

\tikzset{
	->,
	> = {Stealth[length=5pt]},
	node distance = 4cm,
	every state/.style = {circle, line width=0.35mm, minimum size=13pt, inner sep=3pt, font=\footnotesize},
	every node/.style = {font=\footnotesize},
	initial text = $ $,
}

\begin{document}

\author[A.G.R.~Cruz]{Anjelo Gabriel R.~Cruz}
\address[A.G.R.~Cruz]{Institute of Mathematics, University of the Philippines Diliman, 1101 Quezon City, Philippines}
\email{agcruz@math.upd.edu.ph}

\author[M.J.C.~Loquias]{Manuel Joseph C.~Loquias}
\address[M.J.C.~Loquias]{Institute of Mathematics, University of the Philippines Diliman, 1101 Quezon City, Philippines}
\email{mjcloquias@math.upd.edu.ph}

\author[J.M.~Thuswaldner]{J\"org M.~Thuswaldner}
\address[J.M.~Thuswaldner]{Chair of Mathematics and Statistics, University of Leoben, Franz-Josef-Strasse 18, A-8700 Leoben, Austria}
\email{joerg.thuswaldner@unileoben.ac.at}
	
\thanks{The authors are supported by the ANR-FWF grant I 6750. A.G.R.~Cruz is grateful to the University of the Philippines System for 	financial support through its Faculty, REPS, and Administrative Staff Development Program.}

\subjclass[2020]{11A63, 68Q45}

\keywords{matrix number systems, dynamical systems, automata, radix representations}
	
\title[Digit Systems for Expanding Rational  Matrices]{Addition automata and attractors of Digit Systems Corresponding to Expanding Rational Matrices}

\begin{abstract}
Let $A$ be an expanding $2 \times 2$ matrix with rational entries 
and $\Z^2[A]$ be the smallest $A$-invariant $\Z$-module containing $\Z^2$. Let $\D$ be a finite subset of $\Z^2[A]$ which is a complete residue system of $\Z^2[A]/A\Z^2[A]$. The pair $(A,\mathcal{D})$ is called a {\em digit system} with {\em base} $A$ and {\em digit set} $\mathcal{D}$. It is well known that every vector $x \in \Z^2[A]$ can be written uniquely in the form
\[ 
x = d_0 + Ad_1 + \cdots + A^kd_k + A^{k+1}p,
\]
with $k\in \N$ minimal, $d_0,\dots,d_k \in \mathcal{D}$, and $p$ taken from a finite set of {\em periodic elements}, the so-called {\em attractor} of $(A,\mathcal{D})$. If $p$ can always be chosen to be $0$ we say that $(A,\mathcal{D})$ has the {\em finiteness property}.

In the present paper we  introduce finite-state transducer automata which realize the addition of the vectors $\pm(1,0)^\top$ and $\pm(0,1)^\top$ to a given vector $x\in \Z^2[A]$ in a number system $(A,\mathcal{D})$ with collinear digit set. These automata are applied to characterize all pairs $(A,\mathcal{D})$ that have the finiteness property and, more generally, to characterize the attractors of these digit systems. 
\end{abstract}

\maketitle

\section{Introduction}

Let $A$ be an $n \times n$ invertible matrix with integer entries and let $\D$ be a finite subset of $\Z^n$. The  pair $(A,\D)$ is called a \textit{digit system} (or \textit{number system}) with base~$A$ and digit set $\D$.
If a given vector $x \in \Z^n$ admits an expansion of the form
$x = d_0 + Ad_1 + \cdots + A^k d_k$,	
where $k$ is a nonnegative integer and $d_0, d_1, \dots, d_k \in \D$, we say that $x$ has a finite 
expansion in the digit system $(A,\D)$. If every $x \in \Z^n$ can be expressed uniquely in such a form, then $(A,\D)$ is said to have the \textit{finiteness} and \textit{uniqueness properties}. Such digit systems  and their properties have been studied extensively in the literature, for instance in \cite{german}, \cite{katai}, \cite{kovacs}, and \cite{kov-lattice}. Their properties can also be viewed geometrically which leads to a tiling theory for so-called \textit{self-affine tiles}, see for instance~\cite{LW96} and \cite{LW97}. 

The concept of a number system with a matrix base can be extended to rational matrices. Let $A\in\Q^{n\times n}$ be invertible. Then the set of representable vectors is not just $\Z^n$ but
\[
\Z^n[A] = \bigcup_{k=0}^\infty (\Z^n + A\Z^n + \cdots + A^k\Z^n).
\] 
That is, $\Z^n[A]$ consists of all vectors $x \in \Q^n$ of the form
$   x = x_0 + Ax_1 + \cdots + A^\ell x_\ell$, where $\ell$ is a nonnegative integer and $x_0, x_1, \dots, x_\ell \in \Z^n$. The collection $\Z^n[A]$ can also be seen as the smallest nontrivial $A$-invariant $\Z$-module containing $\Z^n$. If $\D \subseteq \Z^n[A]$ is finite we call $(A,\D)$ again a \textit{digit system} with base~$A$ and digit set $\D$. (Note that if $A$ has integer entries, we have $\Z^n[A] =\Z^n$ and everything reduces to the case discussed before.) In \cite{jankauskas} and \cite{jankauskas-ii}, Jankauskas and Thuswaldner gave sufficient and necessary conditions to guarantee the existence of a finite set $\D \subseteq \Z^n[A]$ for which every $x \in \Z^n[A]$ can be expressed in the form 
\begin{equation} \label{eq:expansionD}
x = d_0 + Ad_1 + \cdots + A^k d_k,
\end{equation}
where $k$ is a nonnegative integer and $d_0, d_1, \dots, d_k \in \D$. The pair $(A,\D)$ is then called a digit system having the \textit{finiteness property}. If, in addition, $\D$ is a complete residue system of $\Z^n[A]/A\Z^n[A]$, then the representation of $x$ in (\ref{eq:expansionD}) is unique if $k \in \N$ is chosen minimally. In this case, $(A,\D)$ is said to have the \textit{uniqueness property}. The special case where $A$ is a $1 \times 1$ matrix produces a digit system whose base is a rational number $p/q$, where $p$ and $q$ are coprime integers with $\abs{p} > \abs{q}$. Such a digit system and its automaticity properties have been studied extensively by Akiyama et al.~in \cite{frac-akiyama}, \cite{auto-akiyama}, and \cite{subt-akiyama}.

Assume that the $n \times n$ matrix $A$ is \textit{expanding}, {\it i.e.}, that all eigenvalues of $A$ must have modulus greater than 1, and that the digit set $\D$ is a finite subset of $\Z^n[A]$ which is a complete residue system of $\Z^n[A]/A\Z^n[A]$. In this case one can show that the digit system $(A,\mathcal{D})$ is close to the finiteness property in the sense that it admits a finite {\em attractor} $\mathcal{A}$ such that each $x\in \Z^n[A]$ can be written in the form 
\begin{equation} \label{eq:expansionD}
x = d_0 + Ad_1 + \cdots + A^k d_k + A^{k+1} p,
\end{equation}
where $k$ is a nonnegative integer, $d_0,\dots,d_k \in \mathcal{D}$, and $p\in \mathcal{A}$. In this paper, we investigate digit systems with these properties. In particular, we concentrate on the case where $A$ is a $2 \times 2$ expanding matrix and choose the digit set to be collinear. We introduce finite-state transducer automata which realize the addition of the vectors $\pm(1,0)^\top$ and $\pm(0,1)^\top$ to a given vector $x\in \Z^2[A]$ on the digit expansion \eqref{eq:expansionD} of $x$. These addition automata, which are of interest in their own right, can be used to determine whether $(A,\mathcal{D})$ has the finiteness property and, more generally, to characterize the attractor of $(A,\mathcal{D})$. Indeed, we are able to provide  these automata explicitly and use them to characterize the finiteness property and the attractors for this whole class of digit systems. The shape of our digit sets allows us to relate our matrix digit systems to other kinds of digit systems, namely to canonical number systems (CNS) and shift radix systems (SRS). Canonical number systems have a long history that goes back to Knuth \cite{knuth-art}, and they were first studied systematically by K\'{a}tai and his coworkers in \cite{katai-comp-int} for Gaussian integers, \cite{katai-quad-alg} for real quadratic number fields, and \cite{katai-im-quad} for imaginary quadratic number fields. They were later studied in greater generality in \cite{kov-int-ring}, \cite{kov-int-dom}, \cite{petho}, and in the ``nonmonic'' case that is particularly interesting for us, in \cite{scheier}. On the other hand, shift radix systems were invented in \cite{srs-akiyama} and have been studied more extensively in \cite{berthe} and \cite{srs-kirsch}. They are known to contain nonmonic canonical number systems as special cases; see \cite{berthe}.

\section{Arithmetics of Matrix Digit Systems}

We first recall the theory of matrix digit systems as developed in \cite{jankauskas} and \cite{jankauskas-ii}. Suppose $A \in \Q^{n\times n}$ is expanding. Then $\Z^n[A] / A\Z^n[A]$ is a finite abelian group as was shown for instance in \cite{jankauskas-ii}. Let $\D \subseteq \Z^n[A]$ be a complete system of residues of $\Z^n[A] / A\Z^n[A]$. 
To enable positional arithmetic computations, we also assume that $\D$ contains the zero vector $0 = (0,\dots,0)^\top$. The \textit{digit function} is defined to be the mapping $\d : \Z^n[A] \to \D$ which sends a given vector $x \in \Z^n[A]$ to the unique element $\d(x) \in \D$ satisfying \[x \equiv \d(x) \pmod{A\Z^n[A]},\]
i.e., $x - \d(x) \in A\Z^n[A]$. Thus, $\d$ sends the vector $x$ to its ``remainder" in $\D$. Define the dynamical system $\Phi : \Z^n[A] \to \Z^n[A]$ associated to $\d$ by \[\Phi(x) = A^{-1}\big(x - \d(x)\big).\]
Recall that the \textit{attractor} $\A$ of $\Phi$ is the smallest subset of $\Z^n[A]$ such that $\Phi(\A) \subseteq \A$ and for each $x \in \Z^n[A]$, there exists $k \in \N$ such that $\Phi^k(x) \in \A$. Such a set $\A$ is unique and is denoted by $\A_\Phi$.

It has been shown in \cite[Proposition 2.5]{jankauskas-ii} that $\A_\Phi$ is finite whenever $A$ is expanding. In this case, $\A_\Phi$ consists of all \textit{periodic points} of $\Phi$, i.e., vectors $p \in \Z^n[A]$ such that $\Phi^r(p) = p$ for some $r \in \Z^+$. Moreover, each $x \in \Z^n[A]$ can be written uniquely in the form
\begin{equation} \label{eq:attractor}
    x = \sum_{j=0}^{n-1} A^jd_j + A^{n}p
\end{equation}
for some minimal $n \in \N$, $d_0,\dots,d_{n-1} \in \D$, and $p \in \A_\Phi$. In fact, $n$ is the smallest positive integer for which $\Phi^{n}(x) \in \A$, $d_j = (\d \circ \Phi^j)(x)$ for all $j \in \brc{0,\dots,n-1}$, and $p = \Phi^{n}(x)$.

Because $p$ is a periodic point, there is a smallest positive integer $r$ such that $\Phi^r(p) = p$. Hence there exist unique $b_0,\dots,b_{r-1} \in \D$ such that $p = \sum\limits_{k=0}^{r-1} A^kb_k + A^rp$. Iterating this equation for arbitrarily large $N$ gives
\[
p = \sum_{k=0}^{Nr-1} A^kb_{k \Mod r} + A^{Nr}p.
\]
Hence, by substituting to $p$ in (\ref{eq:attractor}), $x$ may be written in the form 
\begin{equation} \label{eq:padic}
    x = \sum_{j=0}^{n-1} A^jd_j + A^n\sum_{k=0}^{Nr-1} A^kb_{k \Mod r} + A^{Nr+n}p =: (\sigma)_A,
\end{equation}
where $\sigma$ is the infinite word $(b_{r-1} \cdots b_0)^\infty d_{n-1} \cdots d_0$. This expansion is called the \textit{$A$-adic representation} of $x$. As a consequence of (\ref{eq:padic}), each $x \in \Z^n[A]$ corresponds to a unique infinite word of the form $\sigma = (b_{r-1} \cdots b_0)^\infty d_{n-1} \cdots d_0$ over $\D$ for which $(\sigma)_A = x$. 
If $0\in\A_\Phi$, then a vector $x\in\Z^n[A]$ corresponding to the infinite word $0^\infty d_{n-1} \cdots d_0$ can be represented by the finite word $d_{n-1} \cdots d_0$ over $\D$,  and we say that $x$ has a finite expansion.
In particular, if $\A_\Phi=\brc{0}$ then the matrix digit system $(A,\D)$ has the finiteness property.

\section{Statement of the Main Results}
\label{SMR}

For the remainder of this paper, we will focus on the case where $n = 2$. Let us consider a $2 \times 2$ matrix of the form 
\[A = A(\alpha,\beta) =
\begin{bmatrix} 
    0 & -\beta \\ 
    1 & -\alpha 
\end{bmatrix},\]
where $\alpha, \beta \in \Q$. Observe that $A$ is precisely the companion matrix of the polynomial $p(x) = x^2 + \alpha x + \beta \in \Q[x]$. This turns out to be a convenient choice for $A$ as any other nonscalar matrix $B$ with characteristic polynomial $p(x)$ has $A$ as its Frobenius normal form.

Throughout this paper, we assume that $p(x)$ is irreducible over $\Q$. Furthermore, we assume that $A$ is expanding in order to ensure that the attractor of $\Phi$ is finite. We also denote by $c$ the (unique) positive integer such that $cp(x) = cx^2 + ax + b$ is a primitive polynomial with integer coefficients (here, $a = c\alpha$ and $b = c\beta$). We also use $q(x)$ to denote the polynomial $cp(x)$. Because $q(x)$ is irreducible in $\Q[x]$, it follows that $\abs{\Z^2[A]/A\Z^2[A]} = \abs{b}$ \cite[Proposition 2.2]{rossi}.

We begin our discussion with the following result. The set $\D$ defined here will serve as our digit set for the rest of the paper.

\begin{proposition} \label{prop:crs}
Suppose $A = A(\alpha,\beta)$ is the companion matrix of the irreducible polynomial $p(x) = x^2 + \alpha x + \beta \in \Q[x]$. Then the set $\D = \brc{(0,0)^\top, (1,0)^\top, \dots, (\abs{b}-1,0)^\top}$ is a complete residue system of $\Z^2[A]/A\Z^2[A]$.
\end{proposition}

\begin{proof}
Let $v = \sum\limits_{j=0}^m A^j v_j \in \Z^2[A]$ (where each $v_j$ is in $\Z^2$). Since each $v_j$ is an integral linear combination of the canonical basis vectors $\e_1$ and $\e_2$ of $\Q^2$ and $\e_2 = A\e_1$, it follows that
\[v = \sum_{j=0}^\ell A^j (b_j,0)^\top\]
for some $b_j \in \Z$. Because $q(A) = cA^2 + aA + bI = 0_{2\times 2}$, we may take $b_0 \in \brc{0,1,\dots,\abs{b}-1}$ by adding or subtracting $q(A) \,\textbf{e}_1 = 0_{2\times 1}$ to $v$. Thus \[v + A\Z^2[A] = (b_0,0)^\top + A\Z^2[A] \in \D + A\Z^2[A],\] and so $\D$ contains a complete residue system modulo $A\Z^2[A]$. Since $\abs{\D} = \abs{b} = \abs{\Z^2[A]/A\Z^2[A]}$, it follows that $\D$ is a complete residue system of $\Z^2[A]/A\Z^2[A]$.
\end{proof}

The proof of Proposition \ref{prop:crs} also applies to companion matrices of arbitrary size $n$. It similarly holds for when $A$ is a $2 \times 2$ matrix of the form
\[A =
\begin{bmatrix}
	r & s \\
	\varepsilon & t
\end{bmatrix}
\]
where $r \in \Z$, $\varepsilon\in \{1,-1\}$, and $s,t \in \Q$. This is because the vectors $\e_1$ and $(r,\varepsilon)^\top = A\e_1$ are the columns of a $2 \times 2$ integer matrix with determinant $\pm 1$, and thus also constitute an integral basis for the square lattice $\Z^2$.

We now turn our attention to the digit system $(A,\D)$, where $A = A(\alpha,\beta)$ is the companion matrix of $p(x) = x^2 + \alpha x + \beta$ and $\D$ is the digit set defined in Proposition \ref{prop:crs}. Our goal is to determine a transducer automaton that performs the addition of $\pm(1,0)^\top$ and $\pm(0,1)^\top$ in the digit system $(A,\D)$ and to characterize the attractor $\A_\Phi$ of $\Phi$. 

To state our main result, we need to define a \textit{finite letter-to-letter transducer} that realizes the addition of a vector $x \in \Z^2[A]$ by the canonical basis vectors $\e_1$ and $\e_2$ of $\Q^2$. We recall below the definition of a finite letter-to-letter transducer. 

\begin{definition}
A \textit{(finite letter-to-letter) transducer} or \textit{automaton} is an ordered 6-tuple $\mathcal{T} = (\mbf{Q},\mathcal{A},\mathcal{B},\mbf{E},\mbf{I},\mbf{T})$, where:
\begin{itemize}
	\item $\mbf{Q}$ is a finite nonempty set whose elements are called \textit{states};
	\item $\mathcal{A}$ and $\mathcal{B}$ are finite sets of \textit{digits}, called the \textit{input} and \textit{output alphabets}, respectively;
	\item $\mbf{E} \subseteq \mbf{Q} \times \mathcal{A} \times \mathcal{B} \times \mbf{Q}$ is called the \textit{transition relation}; and
	\item $\mbf{I}$ and $\mbf{T}$ are nonempty subsets of $\mbf{Q}$, called the sets of \textit{initial} and \textit{terminal} states, respectively.
\end{itemize}
\end{definition}
A transducer $\mathcal{T}$ can be represented pictorially by means of a directed, edge-labeled graph known as a \textit{transition diagram}. Here, $\mbf{Q}$ is the set of vertices and a directed edge is drawn from vertex $q_1$ to vertex $q_2$ with label $a \mid b$ if and only if $(q_1,a,b,q_2) \in \mbf{E}$. Such an edge is denoted by $q_1 \xrightarrow{a \mid b} q_2$. A sequence of the form
\[
\mathcal{P} : q_0 \xrightarrow{a_0 \mid b_0} q_1 \xrightarrow{a_1 \mid b_1} q_2 \xrightarrow{a_2 \mid b_2} \cdots \xrightarrow{a_{n-1} \mid b_{n-1}} q_n \xrightarrow{a_n \mid b_n} \cdots
\]
where $q_i \in \mbf{Q}$, $(a_i,b_i) \in \mathcal{A} \times \mathcal{B}$, and $(q_i, a_i, b_i, q_{i+1}) \in \mbf{E}$ for all $i \in \N$, is called a \textit{walk} in the transducer $\mathcal{T}$. The infinite word $\mbf{a} = \cdots a_n a_{n-1} \cdots a_1 a_0 \in \mathcal{A}^\N$ is called the \textit{input} of $\mathcal{P}$, and the word $\mbf{b} = \cdots b_n b_{n-1} \cdots b_1 b_0 \in \mathcal{B}^\N$ is called its corresponding \textit{output}. The pair $\mbf{a} \mid \mbf{b}$ is called the \textit{label} of $\mathcal{P}$.

For more detailed material on automata theory and transducers, we refer the reader to \cite{allouche}, \cite{hopcroft}, and \cite{sakarovitch}.

Recall that if $\sigma$ is an infinite word over $\D$ of the form $(b_{r-1} \cdots b_0)^\infty d_{n-1} \cdots d_0$, we denote the element of $\Z^2[A]$ whose digit representation is $\sigma$ by $(\sigma)_A$. If $b_0 = \cdots = b_{r-1} = 0$ and $d_{n-1} \neq 0$, then we identify $\sigma$ with the finite word $d_{n-1} \cdots d_0$. We introduce the following operations:
\begin{align*}
    (\sigma^{\pm P})_A &= (\sigma)_A \pm (1,0)^\top, \\
    (\sigma^{\pm Q})_A &= (\sigma)_A \pm A(c,0)^\top \pm (a-c,0)^\top, \\
    (\sigma^{\pm R})_A &= (\sigma)_A \mp A(c,0)^\top \mp (a,0)^\top, \\
    (\sigma^{\pm S})_A &= (\sigma)_A \pm (c,0)^\top, \\
    (\sigma^{\pm T})_A &= (\sigma)_A \pm A(c,0)^\top \pm (a+c,0)^\top, \\
    (\sigma^{\pm U})_A &= (\sigma)_A \pm (0,1)^\top = (\sigma)_A \pm A(1,0)^\top.
\end{align*}
Using these symbols, we are now ready to define the {\em addition automaton} that will occur in Theorem~\ref{thm:attractor}. We consider the automaton with set of states $\mbf{Q} = \{\pm P, \pm Q, \allowbreak \pm R, \pm S, \pm T, \pm U\}$, with input and output alphabets $\mathcal{A} = \mathcal{B} = \mathcal{D}$, and with sets of initial and terminal states $\mbf{I}=\mbf{T}=\mbf{Q}$. To define the transition relation $\mbf{E}$ between the states, we determine what becomes of the rightmost digit of $\sigma$ whenever the above operations are applied to it. They turn out to be distinct, depending on the values of $\alpha$ and $\beta$. Consequently, the transition relation $\mbf{E}$ varies across different cases and will be formally defined for each one.

We now state our main result.

\begin{theorem} \label{thm:attractor}
    Let $A = A(\alpha,\beta)$ be the companion matrix of the irreducible polynomial $p(x) = x^2 + \alpha x + \beta \in \Q[x]$, $c$ be the unique positive integer for which the polynomial $q(x) := cp(x)$ is a polynomial with coprime integer coefficients, and $a = c\alpha, b = c\beta$. Consider the digit system $(A,\D)$ where \[\D = \brc{(0,0)^\top, (1,0)^\top, \dots, (\abs{b}-1,0)^\top},\] and denote by $\A_\Phi$ the attractor of the dynamical system $\Phi$ induced by $(A,\D)$. Then the attractor $\A_\Phi$ and the addition automata are characterized as follows:
    \begin{enumerate}[label=\arabic*.]
        \item If $0 < \alpha \leq \beta - 1$, then $\A_\Phi = \brc{0}$ and the associated addition automaton is given in Section~\ref{Case1}.
        \item If $0 < -\alpha < \beta - 1$ and $\alpha < -1$, then $\A_\Phi = \brc{(0,0)^\top, (a+c,c)^\top, \dots, K(a+c,c)^\top}$ where $K$ is the largest positive integer such that $K(a+b+c) \leq b-1$. The associated addition automaton is given in Section~\ref{case2}.
        \item If $0 < -\alpha < \beta - 1$ and $\alpha \geq -1$, then $\A_\Phi = \brc{0}$ and the associated addition automaton is given in Section~\ref{case2}.
        \item If $0 < -\alpha < -\beta - 1$, then $\A_\Phi$ contains $\brc{(0,0)^\top, -(a,c)^\top, -(c,0)^\top}$.  The associated addition automaton is given in Section~\ref{case3}.
        \item If $0 < \alpha < -\beta - 1$, then $\A_\Phi$ contains $\brc{(0,0)^\top,-(a+c,c)^\top,\cdots,-K(a+c,c)^\top}$ where $K$ is the largest positive integer such that $K(-a-b-c) \leq -b-1$.  The associated addition automaton is given in Section~\ref{case4}.
    \end{enumerate}
\end{theorem}

We mention that these addition automata contain not only interesting arithmetic information of the underlying digit system $(A,\D)$, but they are also strongly related to the {\em contact graphs} studied in~\cite{Steiner-Thuswaldner:15} and, hence, also contain important information on geometric objects related to $(A,\D)$. Indeed, they characterize the boundary of ``fundamental domains'' related to $(A,\D)$. We do not pursue this here but intend to come back to this in future work.

A similar result as Theorem~\ref{thm:attractor} has been established in \cite{thuswaldner} for expanding integer matrices of the form
\[
M =
\begin{bmatrix}
    a & b \\ \varepsilon & d
\end{bmatrix},
\]
where $a,b,d \in \Z$ and $\varepsilon \in \brc{1,-1}$, with characteristic polynomial $\chi(x) = x^2 + \alpha x + \beta \in \Z[x]$. The attractor $\mathcal{M}_\Phi$ of the dynamical system generated by $\Phi(z)=M^{-1}(z-\delta)$, where $\delta$ is the unique element of $\brc{(0,0)^\top, (1,0)^\top, \dots, (\abs{\beta}-1,0)^\top}$ with $\delta\equiv {z}\pmod{M}$, was determined by constructing a finite letter-to-letter right transducer that realizes the addition by the vectors $(1,0)^\top$ and $(0,1)^\top$. The identity 
\begin{equation} \label{eq:carry}
    \chi(M) = M^2 + \alpha M + \beta I = 0_{2\times 2},   
\end{equation}
which directs any possible carries in the addition by $(1,0)^\top$, features prominently in the construction of such a transducer. It turns out that the form of the attractor $\mathcal{M}_\Phi$ depends solely on the values of $\alpha = -\tr(M)$ and $\beta = \det(M)$.

D.~Knuth introduces a similar automaton in \cite{knuth} for the numeration system with base $-3/2$ in his study of the space $\mathbbm{K} = \R \times \Q_2$, where $\Q_2$ denotes the field of $2$-adic rational numbers.

Our approach in finding the attractor $\mathcal{A}_\Phi$ will be similar to that in \cite{thuswaldner}; this time, we make use of the identity 
\[
q(A) = cA^2 + aA + bI = 0_{2\times 2}
\] in order to direct any carries when we add by $(1,0)^\top$. However, in our more general situation the structure of the occurring transducers is much more involved and a complete construction of these automata seems to be out of reach.

\section{The case $0 < \alpha \leq \beta - 1$}
\label{Case1}

First, assume that $0 < \alpha \leq \beta-1$. Then $b > 0$ and so $\abs{\Z^2[A]/A\Z^2[A]} = b$. Furthermore, $0 < a \leq b-c$ and this implies that $b > c$ and $b > a$. 
Following the notation of \cite{thuswaldner}, $(\nu,0)^\top \in \D$ is identified with $\nu$, and a horizontal bar over a term means that particular term is taken to be one digit.
Direct computation, together with the identity $q(A)=0$, yields the following relations for $\nu \in \brc{0,1,\dots,b-1}$.
\begin{align*}
    (\sigma\nu)^P &=
    \begin{cases}
        \sigma \, \ol{\nu+1}, & 0 \leq \nu < b-1 \\
        \sigma^R \, 0, & \nu = b-1
    \end{cases} \\
    (\sigma\nu)^{-P} &=
    \begin{cases}
        \sigma \, \ol{\nu-1}, & 1 \leq \nu < b \\
        \sigma^{-R} \, \ol{b-1}, & \nu = 0
    \end{cases} \\
    (\sigma\nu)^{\pm Q} &=
    \begin{cases}
        \sigma^{\pm T} \, \ol{\nu \pm a \mp c \pm b}, & 0 \leq \nu \pm a \mp c \pm b < b \\
        \sigma^{\pm S} \, \ol{\nu \pm a \mp c}, & 0 \leq \nu \pm a \mp c < b \\
        \sigma^{\mp Q} \, \ol{\nu \pm a \mp c \mp b}, & 0 \leq \nu \pm a \mp c \mp b < b
    \end{cases} \\
    (\sigma\nu)^{\pm R} &= 
    \begin{cases}
        \sigma^{\pm Q} \, \ol{\nu \mp a \pm b}, & 0 \leq \nu \mp a \pm b < b \\
        \sigma^{\mp S} \, \ol{\nu \mp a}, & 0 \leq \nu \mp a < b
    \end{cases} \\
    (\sigma\nu)^{\pm S} &=
    \begin{cases}
        \sigma \, \ol{\nu \pm c}, & 0 \leq \nu \pm c < b \\
        \sigma^{\pm R} \, \ol{\nu \pm c \mp b}, & 0 \leq \nu \pm c \mp b < b
    \end{cases} \\
    (\sigma\nu)^{\pm T} &=
    \begin{cases}
        \sigma^{\pm S} \, \ol{\nu \pm a \pm c}, & 0 \leq \nu \pm a \pm c < b \\
        \sigma^{\mp Q} \, \ol{\nu \pm a \pm c \mp b}, & 0 \leq \nu \pm a \pm c \mp b < b
    \end{cases}
\end{align*}
In addition, $(\sigma\nu)^{\pm U} = \sigma^{\pm P} \nu$ for each $\nu \in \brc{0,1,\dots,b-1}$ and this holds for any values of $\alpha$ and $\beta$. 
We show the derivation of the relation involving $\sigma^S$. If $0 \leq \nu < b-c$, then
\[
    \big((\sigma\nu)^S\big)_A = A(\sigma)_A + (\nu,0)^\top + (c,0)^\top = A(\sigma)_A + (\nu+c,0)^\top.
\]
By assumption, $c \leq \nu+c < b$, and so $(\sigma\nu)^S = \sigma \, \ol{\nu+c}$. If $b-c \leq \nu < b$, then $\nu+c \geq b$. Thus a carry occurs in the addition by $(c,0)^\top$ which is captured by $(b,0)^\top = -A^2(c,0)^\top - A(a,0)^\top$. Hence, 
\begin{align*}
    \big((\sigma\nu)^S\big)_A &= A(\sigma)_A + (\nu+c,0)^\top \\
    &= A(\sigma)_A + (\nu+c-b,0)^\top -A^2(c,0)^\top - A(a,0)^\top \\
    &= A\big((\sigma)_A - A(c,0)^\top - (a,0)^\top\big) + (\nu+c-b,0)^\top \\
    &= A(\sigma^R)_A + (\nu+c-b,0)^\top.
\end{align*}
Because $0 \leq \nu+c-b < c < b$, we have $\nu+c-b \in \D$. It follows that $(\sigma\nu)^S = \sigma^R \, \ol{\nu+c-b}$, thereby proving the relation. The other identities are proven similarly. The elements of the transition relation~$\mbf{E}$ corresponding to the relation involving $(\sigma\nu)^S$ are $S \xrightarrow{\nu \mid \nu+c} \bullet$, where the black dot denotes a terminal state, when $0 \leq \nu+c < b$, and $S \xrightarrow{\nu \mid \nu+c-b} R$ when $0 \leq \nu + c - b < b$. 

Observe that at most two of the three possibilities for $Q$ and $-Q$ hold, depending on the sign of $a-c$, which in turn depends on whether $\alpha >1$, $\alpha=1$, or $\alpha < 1$. An example of the case when $\alpha>1$ is given by the automaton in Figure~\ref{fig:automaton11}. This models the addition by $(1,0)^\top$ in the digit system $(A, \D)$ when $\alpha = \tfrac{3}{2}$ and $\beta = \tfrac{7}{2}$. In all cases, a vertex labeled $\ol{u}$, where $u \in \brc{P,Q,R,S,T,U}$, is the state corresponding to the operation $-u$. 
Note that we have omitted the vertices labeled $T$ and $\ol{T}$ in Figure~\ref{fig:automaton11}. 
These states play no role in the automaton, as no walk includes $T$ or $\ol{T}$ unless it starts at one of these states.

To get the string $\sigma^u$ for $u \in \brc{\pm P, \pm Q, \pm R, \pm S, \pm T, \pm U}$, one feeds $\sigma$ to the transducer with starting state labeled $u$. The automaton reads the digits of $\sigma$ from right to left, moving along the edges corresponding to their digits. Every edge label is of the form $j \mid k$, so that the digit $j$ is the input of the automaton, and $k$ is the output corresponding to that particular digit.

\begin{figure}[h]
	\begin{center}
		\includegraphics{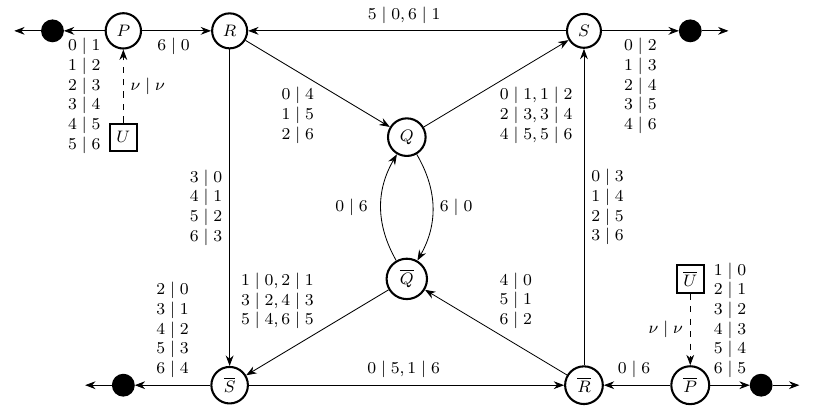}
	\end{center}
	\caption{Transducer that realizes the addition by $(1,0)^\top$ when $\alpha = \tfrac{3}{2}$ and $\beta = \tfrac{7}{2}$.}
	\label{fig:automaton11}
\end{figure}

To add or subtract by $(1,0)^\top$, one starts at the state labeled $P$ or $\ol{P}$, respectively. Analogously, $U$ and $\ol{U}$ are the respective starting points for the addition and subtraction of $(0,1)^\top$. If the automaton reaches the terminal states indicated by the black dots, then all remaining digits of the input string are copied into the output string.

Suppose $x \in \Z^2[A]$ has a digit expansion represented by the finite string $\sigma$ and is fed into the transducer starting at state $P$. If the terminal state is reached after reading the (finitely many) nonzero digits of~$\sigma$, then the remaining digits (all zero) 
are copied into the output string $\sigma^P$. Thus, the digit representation of the vector $x + (1,0)^\top$ is finite.

If we end up in one of the other states, then we consider the remaining digits of~$\sigma$ (i.e., the leading zeroes) and investigate the paths of the transducer whose inputs are all zero. Since all such walks lead to a terminal state after at most four digit readings, it follows that the output string $\sigma^P$ has only finitely many nonzero entries. 
This means that the digit representation of the vector $x + (1,0)^\top$ is still finite.
Hence, the addition by $(1,0)^\top$ produces finite strings from finite strings.

Similar arguments follow for the subtraction of $(1,0)^\top$ as well as the addition and subtraction of $(0,1)^\top$. Hence, because $(0,0)^\top \in \Z^2[A]$ has a finite digit expansion, then so do $(n,0)^\top$, $(0,m)^\top$, and $A^k(m,n)^\top$ for any $m,n \in \Z$ and $k \in \N$. It follows that $\A_\Phi = \brc{0}$.

An example of an automaton for the case where $\alpha < 1$ is provided in Figure~\ref{fig:automaton12}. It models the addition by $(1,0)^\top$ in the digit system $(A, \D)$ when $\alpha = \frac{1}{3}$ and $\beta = \frac{3}{2}$.

\begin{figure}[h]
\begin{center}
    \includegraphics{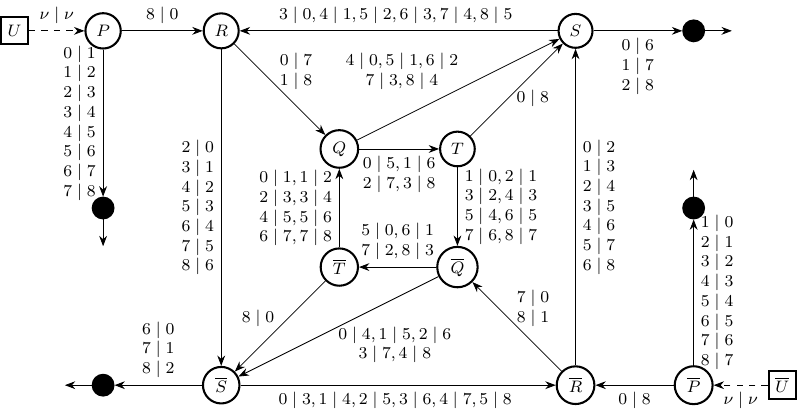}
\end{center}
\caption{Transducer that realizes the addition by $(1,0)^\top$ when $\alpha = \tfrac{1}{3}$ and $\beta = \tfrac{3}{2}$.}
\label{fig:automaton12}
\end{figure}

\section{The case $0 < -\alpha < \beta - 1$}
\label{case2}

The automata for the case $0 \leq -\alpha \leq \beta - 1$ are different compared to the previous case. Because $\alpha < 0$, some of the relations from the previous case have changed, and self-loops can occur. We have the following relations for this case.
\begin{align*}
    (\sigma\nu)^P &=
    \begin{cases}
        \sigma \, \ol{\nu+1}, & 0 \leq \nu < b-1 \\
        \sigma^R \, 0, & \nu = b-1
    \end{cases} \\
    (\sigma\nu)^{-P} &=
    \begin{cases}
        \sigma \, \ol{\nu-1}, & 1 \leq \nu < b \\
        \sigma^{-R} \, \ol{b-1}, & \nu = 0
    \end{cases} \\
    (\sigma\nu)^{\pm Q} &=
    \begin{cases}
        \sigma^{\pm T} \, \ol{\nu \pm a \mp c \pm b}, & 0 \leq \nu \pm a \mp c \pm b < b \\
        \sigma^{\pm S} \, \ol{\nu \pm a \mp c}, & 0 \leq \nu \pm a \mp c < b
    \end{cases} \\
    (\sigma\nu)^{\pm R} &= 
    \begin{cases}
        \sigma^{\mp T} \, \ol{\nu \mp a \mp b}, & 0 \leq \nu \mp a \mp b < b \\
        \sigma^{\mp S} \, \ol{\nu \mp a}, & 0 \leq \nu \mp a < b
    \end{cases} \\
    (\sigma\nu)^{\pm S} &=
    \begin{cases}
        \sigma \, \ol{\nu \pm c}, & 0 \leq \nu \pm c < b \\
        \sigma^{\pm R} \, \ol{\nu \pm c \mp b}, & 0 \leq \nu \pm c \mp b < b
    \end{cases} \\
    (\sigma\nu)^{\pm T} &=
    \begin{cases}
        \sigma^{\pm T} \, \ol{\nu \pm a \pm c \pm b}, & 0 \leq \nu \pm a \pm c \pm b < b \\
        \sigma^{\pm S} \, \ol{\nu \pm a \pm c}, & 0 \leq \nu \pm a \pm c < b \\
        \sigma^{\mp Q} \, \ol{\nu \pm a \pm c \mp b}, & 0 \leq \nu \pm a \pm c \mp b < b
    \end{cases}
\end{align*}
Just like in the previous case, at most two of the three possibilities for $T$ and $-T$ hold, depending on the sign of $a+c$, i.e., whether $\alpha > -1$, $\alpha=-1$, or $\alpha< -1$.
Figure~\ref{fig:automaton21} illustrates a case where $\alpha<-1$.
The states $Q$ and $\ol{Q}$ are not included in the figure.
Although edges emanate from these vertices, they are inaccessible in the automaton. 

Unlike the previous diagrams, the automaton in Figure~\ref{fig:automaton21} has a self-loop at the (accessible) state $T$ with $0$ as its input digit. Hence a vector $x \in \Z^2[A]$ with a finite $A$-adic representation $\sigma$ may be sent to a vector with an infinite $A$-adic representation $\sigma'$ if the walk produced by the digits of $\sigma$ ends up at $T$. It follows that $\A_\Phi$ contains at least one nonzero element.

To illustrate, the vector $x = (8,2)^\top$ corresponds to the digit string $\sigma = 0^\infty 28$. When $\sigma$ is fed into the automaton of Figure~\ref{fig:automaton21} with initial state $P$, we obtain the output string $\sigma^P = 8^\infty 5660$. Thus, $(\sigma^P)_A = (9,2)^\top = x+(1,0)^\top$.

\begin{figure}[h]
    \begin{center}
    \includegraphics{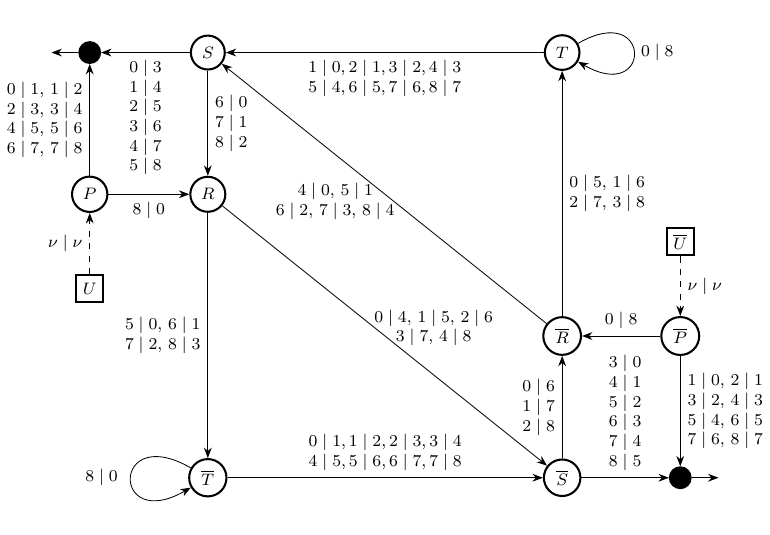}
	\end{center}
    \caption{Transducer that realizes the addition by $(1,0)^\top$ when $\alpha = -\tfrac{4}{3}$ and $\beta = 3$.}
    \label{fig:automaton21}
\end{figure}

Suppose $\alpha<-1$.
In order to obtain the elements of $\A_\Phi$, set $\gamma = a+b+c$ and consider the self-loop at $T$ labeled $0 \mid \gamma$. 

First, suppose the input string $\sigma$ to the automaton $\sigma$ is finite, i.e., $\sigma = 0^\infty d_{n-1} \cdots d_1 d_0$. 
Applying the automaton to $\sigma$ (with any initial state) yields the following two possibilities. 
\begin{itemize}
    \item The walk produced by $\sigma$ ends at a terminal state. Thus the output string is also of the form $0^\infty u$, for some finite word $u$ over $\D$. 
    \item The walk produced by $\sigma$ ends at the self loop at $T$. This time the output string is of the form $\gamma^\infty w$, for some finite word $w$ over $\D$.
\end{itemize}
Hence $(0^\infty)_A = 0$ and $(\gamma^\infty)_A$ are possible periodic points of $\Phi$. 

Using a similar argument, applying the automaton to an input string of the form $\gamma^\infty d_{n-1} \cdots d_1d_0$ yields an output string that is either finite or is periodic of the form $\gamma^\infty u$ or $(\ol{2\gamma})^\infty w$, for some finite words $u, w$ over $\D$.


In general, let $\ell \in \brc{0,1,\dots,K}$,  where $K$ is the largest positive integer such that $K\gamma \leq b-1$.
Then any input string of the form $(\ol{\ell\gamma})^\infty d_{n-1} \cdots d_1d_0$ will produce an output string having one of the following forms:
\begin{itemize}
    \item $(\ol{\ell\gamma})^\infty u$, for some finite word $u$ over $\D$, if the trail produced by $\sigma$ ends at a terminal state;
    \item $(\ol{(\ell-1)\gamma})^\infty v$, for $\ell > 0$ and some finite word $v$ over $\D$, if the trail produced by $\sigma$ ends at the self loop at~$\ol{T}$;
    \item $(\ol{(\ell+1)\gamma})^\infty w,$ for $\ell < K$ and some finite word $w$ over $\D$, if the trail produced by $\sigma$ ends at the self loop at~$T$.
\end{itemize}
Starting with $(0,0)^\top$ and repeatedly applying the automaton with initial states $\pm P$ or $\pm U$ allows us to exhaust every element of $\Z^2[A]$.
Consequently, any $x \in \Z^2[A]$ has a representation of the form \[x = \sum_{j=0}^{n-1} A^{j}d_j + A^n p_{\ell}\] where $p_\ell = \big((\ol{\ell\gamma})^\infty\big)_A$, $\ell \in \brc{0,1,\dots,K}$. The points $p_\ell$ are the only periodic points of $\Phi$.
Each $p_\ell$ is obtained by first running the automaton on $(0,0)^\top$ with initial state $T$, then repeatedly applying the same process to each output $\ell$ times.
This gives
\[p_\ell = A(\ell c,0)^\top + (\ell(a+c),0)^\top = \ell(a+c,c)^\top.\]
Hence $\A_\Phi = \brc{(0,0)^\top,(a+c,c)^\top,\dots,K(a+c,c)^\top}$.

An example of an automaton for $\alpha \geq -1$ is given in Figure~\ref{fig:automaton22}. 
Unlike the case when $\alpha<-1$, its transition diagram has no self-loops.
Similar to Figures \ref{fig:automaton11} and \ref{fig:automaton12}, it maps finite representations to finite representations. Thus by the previous arguments, $\A_\Phi = \brc{0}$.

\begin{figure}[h]
\begin{center}
	\includegraphics{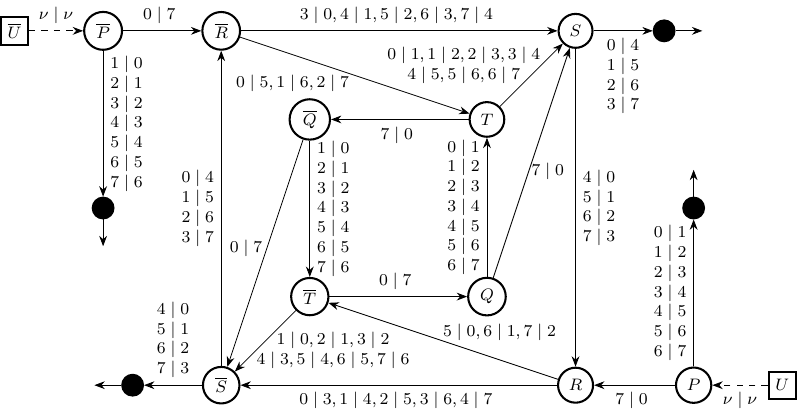}
\end{center}
\caption{Transducer that realizes the addition by $(1,0)^\top$ when $\alpha = -\tfrac{3}{4}$ and $\beta = 2$.}
\label{fig:automaton22}
\end{figure}

\section{The case $0 < -\alpha < -\beta - 1$}
\label{case3}
If $0 < -\alpha \leq -\beta - 1$, then $b < 0$ and so $\abs{\Z^2[A]/A\Z^2[A]} = -b$. In contrast to the automata in the former two cases, the states $T$ and $\ol{T}$ play no role here.
Also, the transducers have the same form regardless of the value of $\alpha$. The relations for this case are given as follows. 
\begin{align*}
    (\sigma\nu)^P &=
    \begin{cases}
        \sigma \, \ol{\nu+1}, & 0 \leq \nu < -b-1 \\
        \sigma^{-R} \, 0, & \nu = -b-1
    \end{cases} \\
    (\sigma\nu)^{-P} &=
    \begin{cases}
        \sigma \, \ol{\nu-1}, & 1 \leq \nu < -b \\
        \sigma^{R} \, \ol{-b-1}, & \nu = 0
    \end{cases} \\
    (\sigma\nu)^{\pm Q} &=
    \begin{cases}
        \sigma^{\mp Q} \, \ol{\nu \pm a \mp c \mp b}, & 0 \leq \nu \pm a \mp c \mp b < -b \\
        \sigma^{\pm S} \, \ol{\nu \pm a \mp c}, & 0 \leq \nu \pm a \mp c < -b
    \end{cases} \\
    (\sigma\nu)^{\pm R} &= 
    \begin{cases}
        \sigma^{\pm Q} \, \ol{\nu \mp a \pm b}, & 0 \leq \nu \mp a \pm b < -b \\
        \sigma^{\mp S} \, \ol{\nu \mp a}, & 0 \leq \nu \mp a < -b
    \end{cases} \\
    (\sigma\nu)^{\pm S} &=
    \begin{cases}
        \sigma \, \ol{\nu \pm c}, & 0 \leq \nu \pm c < -b \\
        \sigma^{\mp R} \, \ol{\nu \pm c \pm b}, & 0 \leq \nu \pm c \pm b < -b
    \end{cases}
\end{align*}
An example of an automaton for this case is depicted in Figure~\ref{fig:automaton3}.
\begin{figure}[h]
\begin{center}
	\includegraphics{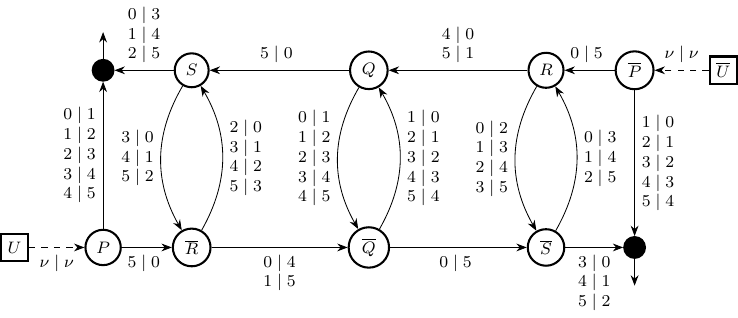}
\end{center}
\caption{Transducer that realizes the addition by $(1,0)^\top$ when $\alpha = -\tfrac{2}{3}$ and $\beta = -2$.}
\label{fig:automaton3}
\end{figure}

Suppose $x \in \Z^2[A]$ has a finite digit representation, say $\sigma = 0^\infty d_{n-1} \cdots d_1 d_0$. Then the path in the automaton defined by $\sigma$ ends at either a terminal state, or ends up at the cycle $-S \to R \to -S$.
In the former, the output word is of the form $0^\infty u$ for some finite word $u$ over $\D$; in the latter, the output word is of the form $\big(\,\ol{(-b-c)} \, \ol{(-a)}\,\big)^\infty \, w$ or $\big(\,\ol{(-a)} \, \ol{(-b-c)}\,\big)^\infty \, w$ for some finite word $w$ over $\D$. Thus
\[
(0^\infty)_A = 0, \big(\big(\ol{(-b-c)} \, \ol{(-a)}\big)^\infty\big)_A, \text{ and } \big(\ol{(-a)} \, \big(\ol{(-b-c)}\big)^\infty\big)_A
\]
are periodic points of $\Phi$. The second and third points are obtained by applying the automaton to the representation of $0$ starting at state $R$ and state $-S$, respectively. The former state subtracts $0$ by $A(c,0)^\top + (a,0)^\top = (a,c)^\top$, while the latter subtracts $0$ by $(c,0)^\top$. Hence $(0,0)^\top, -(a,c)^\top, -(c,0)^\top \in \A_\Phi$.

It should be noted that $(0,0)^\top,$ $-(a,c)^\top,$ and $-(c,0)^\top$ are not the only periodic points of $\Phi$. Other periodic points occur depending on the values of $\alpha$ and $\beta$. 
For instance, the path in the automaton defined by an infinite word of the form $\sigma = \big(\ol{(-b-c)} \, \ol{(-a)}\big)^\infty u$ (where $u$ is a finite word over $\D$) may end up at a terminal state or at the cycle $-S \to R \to -S$ as in the preceding paragraph. The former produces an output having the same form as $\sigma$, while the latter produces an output of the form $0^\infty w$, with $w$ a finite word over $\D$.

However, the path defined by $\sigma$ may also end up at the cycle $-Q \to Q \to -Q$; this occurs whenever $-a \geq a-b-c$ and $-b-c < -a+c$, i.e., whenever $-\beta - 1 \leq -2\alpha$ and $\beta - \alpha > -2$. In this case, the output word is of the form $\big(\ol{(a-2b-2c)} \, \ol{(-2a+b+c)}\big)^\infty w$ where $w$ is a finite word over $\D$. This output gives rise to another periodic point, namely $\big(\big(\ol{(a-2b-2c)} \, \ol{(-2a+b+c)}\big)^\infty\big)_A$. It is obtained by feeding the representation of $-(a,c)^\top$ into the automaton starting at state $-Q$, i.e., adding $-(a,c)^\top$ by $-A(c,0)^\top - (a-c,0)^\top$. Thus, \[-(a,c)^\top - A(c,0)^\top - (a-c,0)^\top = -(2a-c,2c)^\top\]
is a periodic point of $\Phi$ whenever $-\beta - 1 \leq -2\alpha$ and $\beta - \alpha > -2$. An example of this particular case is given in Figure \ref{fig:automaton3}, where the infinite word $\big(\ol{(-b-c)} \, \ol{(-a)}\big)^\infty u = (32)^\infty 05$ produces the output $\big(\ol{(a-2b-2c)} \, \ol{(-2a+b+c)}\big)^\infty w = (41)^\infty 40$.

\section{The case $0 < \alpha < -\beta - 1$}
\label{case4}

For this case, the states $Q$ and $-Q$ play no role in the transducer. The relations satisfied by the other states are given as follows.
\begin{align*}
    (\sigma\nu)^P &=
    \begin{cases}
        \sigma \, \ol{\nu+1}, & 0 \leq \nu < -b-1 \\
        \sigma^{-R} \, 0, & \nu = -b-1
    \end{cases} \\
    (\sigma\nu)^{-P} &=
    \begin{cases}
        \sigma \, \ol{\nu-1}, & 1 \leq \nu < -b \\
        \sigma^{R} \, \ol{-b-1}, & \nu = 0
    \end{cases} \\
    (\sigma\nu)^{\pm R} &= 
    \begin{cases}
        \sigma^{\mp T} \, \ol{\nu \mp a \mp b}, & 0 \leq \nu \mp a \mp b < -b \\
        \sigma^{\mp S} \, \ol{\nu \mp a}, & 0 \leq \nu \mp a < -b
    \end{cases} \\
    (\sigma\nu)^{\pm S} &=
    \begin{cases}
        \sigma \, \ol{\nu \pm c}, & 0 \leq \nu \pm c < -b \\
        \sigma^{\mp R} \, \ol{\nu \pm c \pm b}, & 0 \leq \nu \pm c \pm b < -b
    \end{cases} \\
    (\sigma\nu)^{\pm T} &=
    \begin{cases}
        \sigma^{\pm T} \, \ol{\nu \pm a \pm c \pm b}, & 0 \leq \nu \pm a \pm c \pm b < -b \\
        \sigma^{\pm S} \, \ol{\nu \pm a \pm c}, & 0 \leq \nu \pm a \pm c < -b
    \end{cases}
\end{align*}

An example of a transducer depicting this case is given in Figure~\ref{fig:automaton4}. 

\begin{figure}[h]
\begin{center}
	\includegraphics{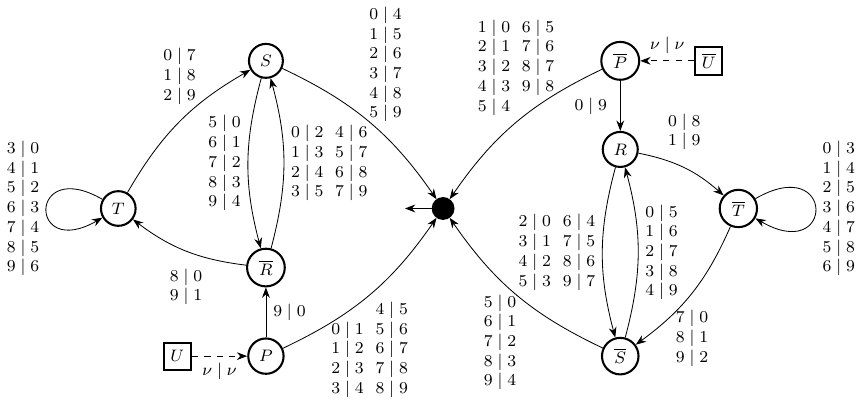}
\end{center}    
\caption{Transducer that realizes the addition by $(1,0)^\top$ when $\alpha = \tfrac{2}{5}$ and $\beta = -2$.}
\label{fig:automaton4}
\end{figure}

Analogous to the case in Figure~\ref{fig:automaton21} where $0 < -\alpha < \beta - 1$ and $\alpha<-1$, the transducer has a self-loop at $-T$ whose input is $0$ and whose output is $\gamma := -a-b-c$. As such, the argument in finding the periodic points is also similar to the aforementioned case: if $K$ is the largest positive integer such that $K\gamma \leq -b-1$ and $\gamma \geq c$, then the only possible periodic points of $\Phi$ are $\big(\ol{(\ell\gamma)}^\infty\big)_A$ where $\ell \in \brc{0,1,\dots,K}$. These points are obtained by feeding the representation of $0$ to the automaton with starting state $-T$ and iterating $\ell$ times. Computing for these points gives $(0,0)^\top,-(a+c,c)^\top,\dots,-K(a+c,c)^\top \in \A_\Phi$. 

Similar to the previous case, $\A_\Phi$ may contain points other than those obtained in the preceding discussion. For instance, suppose $K > 2$. Then the path in the automaton corresponding to an infinite word of the form $\sigma = \ol{(2\gamma)}^\infty u$ where $u$ is a finite word over $\D$ may end up at the loop $S \to -R \to S$ (particularly when $-b-c \leq 2\gamma < -a-b$). This produces an output word of the form $\big(\ol{(\gamma-a)} \, \ol{(\gamma-b-c)}\big)^\infty w$ where $w$ is a finite word over $\D$. An example illustrating this situation is given in Figure \ref{fig:automaton4}, where the infinite word $(2\gamma)^\infty u = 6^\infty 9$ produces the output word $\big(\ol{(\gamma-a)} \, \ol{(\gamma-b-c)}\big)^\infty w = (18)^\infty 0$.

\section{Relationship between matrix and polynomial digit systems}
\label{MPDS}

Our construction of the automata in the previous sections showed that the characteristic polynomial of $A = A(\alpha,\beta)$ mattered more than the actual matrix $A$ itself. Indeed, we used solely the matrix equation \[q(A) = cA^2 + aA + bI = 0_{2\times 2}\] when dealing with carries in the addition by $(1,0)^\top$. Hence, such computations hold for any matrix $B$ whose Frobenius normal form is $A$ (as $B$ will then have the same characteristic polynomial as $A$) and for which the digit set $\D = \{(0,0)^\top, (1,0)^\top, \dots, (\abs{b}-1,0)^\top\}$ is also a complete set of coset representatives of $\Z^2[B]/B\Z^2[B]$. This allows us to work only with the polynomial $q(x)$ instead of the matrix $A$. 

We now turn to the question of relating the matrix digit system $(A,\D)$ with the so-called digit system induced by the ring $\Z[x]/q\Z[x]$, as introduced by Scheicher et al.~in \cite{scheier}.

We first recall the necessary concepts in \cite{scheier}. While the results in \cite{scheier} hold for any commutative ring $R$ with unity, we focus our attention on the case where $R = \Z$. Let $d \in \Z^+$ and \[P = p_dx^d + \cdots + p_1x + p_0 \in \Z[x]\] such that $p_d$ and $p_0$ are nonzero. Let $X = x + P\Z[x]$ be the image of $x$ under the canonical epimorphism from $\Z[x]$ to $\Z[x]/P\Z[x] =: \sR$ and $\sN$ be a complete residue system of $\sR/(X)$. Define the map $D_\sN : \sR \to \sN$ where any $r \in \sR$ is mapped to the unique element $D_\sN(r)$ of $\sN$ satisfying
\[
r \equiv D_\sN(r) \pmod X,
\]
i.e., $r - D_\sN(r) \in (X)$. Let $T : \sR \to \sR$ be the associated dynamical system defined by \[T(r) = \dfrac{r - \D_\sN(r)}{X}.\]
The triple $(\sR,X,\sN)$ is called the {\em digit system over $R$} defined by $P$ and $\sN$.

In all that follows, we assume that the polynomial $P$ is primitive and identify any $n \in \Z$ with its image under the canonical epimorphism from $\Z$ to $\sR$. Observe that any $r \in \sR$ can be written in the form
\begin{equation} \label{eq:coset}
    r = \sum_{i=0}^k r_i X^i
\end{equation}
where $r_0,\ldots,r_k \in \Z$ and $X = x + P\Z[x]$. Furthermore, $P(X) = 0$ in $\Z[x]/P\Z[x]$. Hence, we may take $r_0 \in \brc{0,1,\dots,\abs{P(0)}-1}$ by adding or subtracting $P(X) = 0$ to $r$ in (\ref{eq:coset}). It follows that $\sN = \brc{0,1,\dots,\abs{P(0)}-1}$ contains a complete residue system of $\sR/(X)$ (and we will also see later that $\sN$ is in fact a complete residue system).

We begin by stating \cite[Lemma 4.1]{scheier} in our current setting. It is a classical result that allows us to represent $r$ in a ``canonical" form similar to (\ref{eq:coset}) that is unique under certain restrictions.

\begin{lemma} \label{lem:cosetrep}
Let $P = p_dx^d + \cdots + p_1x + p_0 \in \Z[x]$ be a primitive polynomial of positive degree $d$. For each $r \in \Z[x]/P\Z[x]$, there exist unique $r_0,\dots,r_k \in \Z$, with $k \in \N$ minimal and $r_d,\dots,r_k \in \brc{0,1,\dots,\abs{p_d}-1}$ whenever $k \geq d$, such that
\[
r = \sum_{i=0}^k r_iX^i
\]
where $X = x + P\Z[x]$.
\end{lemma}

In a similar vein, Rossi et al.~established an analogue of Lemma \ref{lem:cosetrep} 
for matrix digit systems in the proof of \cite[Proposition 2.2]{rossi}. We state the result formally below, assuming once more that $P \in \Z[x]$ is primitive of positive degree $d$.

\begin{lemma} \label{lem:companionrep}
Let $C \in \Q^{d \times d}$ be the companion matrix of some $q \in \Q[x]$ such that $P = cq$ for some $c \in \Z$. Then every $v \in \Z^d[C]$ can be written uniquely in the form
\[v = \sum_{i=0}^\ell b_i C^i \e_1\]
where $\e_1=(1,0,\ldots,0)^\top\in\Q^d$, $\ell \in \N$ is minimal, $b_0,\dots,b_\ell \in \Z$ and $b_d, \dots, b_\ell \in \brc{0,1,\dots,\abs{p_d} - 1}$ if $d \leq \ell$. 
\end{lemma}

Lemma \ref{lem:companionrep} also lets us represent any vector $v \in \Z^d[C]$ in a unique ``canonical" form with conditions similar to those of Lemma \ref{lem:cosetrep}. Using Lemmas \ref{lem:cosetrep} and 
\ref{lem:companionrep}, Rossi et al.~defined the mapping $h : \Z^d[C] \to \Z[x]/P\Z[x]$ where
\[
h\pren{\sum_{i=0}^\ell b_i C^i \e_1} = \sum_{i=0}^\ell b_i X^i.
\]
Lemma \ref{lem:companionrep} guarantees that $h$ is well defined, and Lemma \ref{lem:cosetrep} shows that $h$ is bijective. In fact, it can be shown that $h$ is a $\Z$-module isomorphism, and so $\Z^d[C] \cong \Z[x]/P\Z[x]$ as $\Z$-modules. Furthermore, the restriction of $h$ to $C\Z^d[C]$ is an isomorphism from $C\Z^d[C]$ onto $(X)$, and hence $\Z^d[C] / C\Z^d[C] \cong \sR/(X)$. Because $\abs{\Z^d[C]/C\Z^d[C]}=\abs{\sR/(X)}=\abs{P(0)}$, it follows that the previously defined subset $\sN$ of $\sR$ is indeed a complete residue system of $\sR/(X)$.

Let us consider once more the companion matrix $A = A(\alpha,\beta)$ of $p = x^2 + \alpha x + \beta$ and $c \in \Z^+$ such that $q = cp = cx^2 + ax + b$ is primitive with integer coefficients. By the preceding discussion, the $\Z$-modules $\Z^2[A]$ and $\Z[x]/q\Z[x]$ are isomorphic via the mapping $h : \Z^2[A] \to \Z[x]/q\Z[x]$. We use $h$ to establish a relationship between the digit systems $(A,\D)$ and $(\sR, X, \sN)$ where $\sR = \Z[x]/q\Z[x]$,
\[
    \D = \brc{(0,0)^\top, (1,0)^\top, \dots, (\abs{b}-1,0)^\top}, \text{ and }
    \sN = \brc{0, 1, \dots, \abs{b}-1}.
\]

We consider the digit functions $\d : \Z^2[A] \to \D$ and $D_\sN : \sR \to \sN$, together with their associated dynamical systems $\Phi : \Z^2[A] \to \Z^2[A]$ and $T : \sR \to \sR$ as defined previously. Let $v = \sum\limits_{i=0}^\ell b_iA^i \e_1 \in \Z^2[A]$ where $\ell \in \N$ is minimal, $b_0,\dots,b_\ell \in \Z$ and $b_2,\dots,b_\ell \in \brc{0,1,\dots,c-1}$ whenever $\ell \geq 2$. Then $h(v) = \sum\limits_{i=0}^\ell b_i X^i$. Furthermore,  $\d(v) = (b_0',0)^\top$ and $D_\sN(h(v)) = b_0'$ where $b_0' \in \sN$ such that $b_0' \equiv b_0 \pmod{b}$. Suppose $k \in \Z$ such that $b_0 - b_0' = kb$. Then
\begin{align*}
    \Phi(v) &= A^{-1} [v - \d(v)] = A^{-1} \brk{\sum_{i=1}^\ell b_iA^i \e_1 + (b_0 - b_0') \e_1} \\
    &= A^{-1} \brk{\sum_{i=1}^\ell b_iA^i \e_1 + kb\e_1} = A^{-1} \brk{\sum_{i=1}^\ell b_iA^i \e_1 - kcA^2\e_1 - kaA\e_1} \\
    &= A^{-1} \brk{\sum_{i=3}^\ell b_iA^i \e_1 + (b_2-kc)A^2\e_1 + (b_1-ka)A\e_1} \\
    &= \sum_{i=3}^\ell b_iA^{i-1} \e_1 + (b_2-kc)A\e_1 + (b_1-ka)\e_1,
\end{align*}
where $b_3,\dots,b_\ell \in \brc{0,1,\dots,c-1}$. Thus,
\[
(h \circ \Phi)(v) = \sum_{i=3}^\ell b_i X^{i-1} + (b_2 - kc) X + (b_1 - ka).
\]
On the other hand,
\begin{align*}
    T(h(v)) &= \dfrac{h(v) - D_\sN(h(v))}{X} = \dfrac{\sum\limits_{i=1}^\ell b_i X^i + (b_0 - b_0')}{X} = \dfrac{\sum\limits_{i=1}^\ell b_i X^i + kb}{X} \\
    &= \dfrac{\sum\limits_{i=1}^\ell b_i X^i - kcX^2 - kaX}{X} = \dfrac{\sum\limits_{i=3}^\ell b_i X^i +(b_2-kc)X^2 + (b_1-ka)X}{X} \\
    &= \sum\limits_{i=3}^\ell b_i X^{i-1} +(b_2-kc)X + (b_1-ka) = (h \circ \Phi)(v).
\end{align*}
Thus, we have proven the following result.
\begin{theorem} \label{thm:comm}
Let $A \in \Q^{2\times 2}$ be the companion matrix of the irreducible polynomial $p(x) = x^2 + \alpha x + \beta$, and $c \in \Z^+$ such that $q(x) = cp(x) = cx^2 + ax + b$ is a primitive polynomial with integer coefficients. Suppose $h$, $\D$, $\sN$, $\Phi$, and $T$ are as defined previously. Then
\[T \circ h = h \circ \Phi.\]
\end{theorem}

Hence, in order to obtain the mapping $\Phi$, and subsequently the digits of a given vector $v$, one simply conjugates the dynamical system $T$ by the isomorphism $h$. This implies that the digits in the matrix system $(A,\D)$ may be obtained from the system $(\sR,X,\sN)$ via conjugation. Observe that Theorem \ref{thm:comm} also holds true for a companion matrix of any size $d$ with the same proof.

\section{Shift radix systems and special cases}

The notion of shift radix systems was introduced formally in 2005 by Akiyama et al.~in \cite{srs-akiyama}, in which they established several of its properties and its connections to existing number systems. In particular, shift-radix systems are intimately related to the so-called beta-expansions introduced in 1957 by Renyi \cite{renyi}. They have also been used to characterize when a given polynomial with integer coefficients is the base of what is called a canonical number system, of which \cite{scheier} is a generalization (see \cite{cns-petho} for more details). 

Berth\'{e} et al.~in \cite{berthe} also associated certain collections of fractal tiles with shift radix systems and explored some of their basic properties. For a full survey on shift radix systems, their properties, relationships with other numeration systems, and their geometric aspects, we refer the reader to \cite{srs-kirsch}.

We now use certain regions associated with shift-radix systems to address cases that were not covered in the proof of Theorem \ref{thm:attractor}. We begin by recalling the definition of a shift radix system.

\begin{definition}
Let $d$ be a positive integer and $\r = (r_0, r_1, \dots, r_{d-1}) \in \R^d$. The \textit{shift-radix system} associated to $\r$ is the mapping $\tau_{\r} : \Z^d \to \Z^d$ defined by
\[
\tau_{\r}(\bmath{z}) = (z_1, z_2, \dots, z_{d-1}, -\flr{\bmath{rz}}),
\]
where $\bmath{z} = (z_0, z_1, \dots, z_{d-1}) \in \Z^d$, $\flr{\cdot}$ is the floor function, and $\r\bmath{z}$ is the usual inner product of $\bmath{r}$ and $\bmath{z}$ in $\R^d$. The system $\tau_{\bmath{r}}$ is said to have the \textit{finiteness property} if for each $\bmath{z} \in \Z^d$, there exists $k \in \N$ such that $\tau_{\bmath{r}}^k(\bmath{z}) = \bmath{0}$. The following sets are associated to the SRS $\tau_{\bmath{r}}$:
\begin{align*}
    \D_d &:= \brc{\bmath{r} \in \R^d \mid \text{ for all } \bmath{z} \in \Z^d, \text{ there exist } k,\ell \in \N \text{ such that } \tau_{\r}^k(\bmath{z}) = \tau_{\r}^{k+\ell}(\bmath{z})} \\
    \D_d^{(0)} &:= \brc{\bmath{r} \in \R^d \mid \tau_{\r} \text{ has the finiteness property}}.
\end{align*}
\end{definition}
Observe that $\D_d$ consists of all vectors $\bmath{r} \in \R^d$ for which the sequence $\big(\tau_k(\bmath{z})\big)_{k=1}^\infty$ is eventually periodic for all $\bmath{z} \in \Z^d$. Clearly, $\D_d^{(0)} \subseteq \D_d$. In fact, \cite[Section 4]{srs-akiyama} shows that $\D_d^{(0)}$ can be obtained by ``cutting out" polyhedra from $\D_d$. Figure \ref{fig:dapprox} shows an approximation of $\D_2^{(0)}$.

\begin{figure}[h]
    \centering
    \includegraphics[scale=0.4]{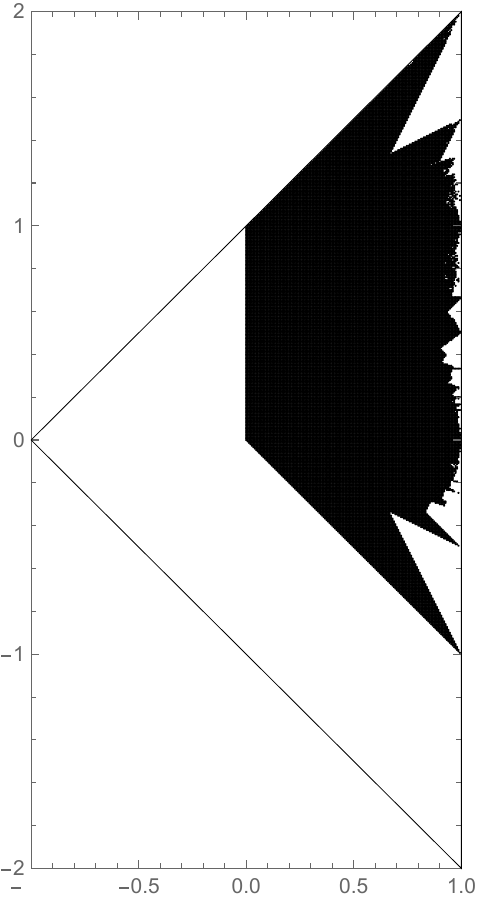}
    \caption{Graphs of $\D_2$ (white triangle) and an approximation of $\D_2^{(0)}$ (black region).}
    \label{fig:dapprox}
\end{figure}

Algebraic integers $\beta>1$ with the finiteness property are characterized in terms of $\D_{d-1}^{(0)}$ in \cite[Theorem 2.1]{srs-akiyama}, while CNS polynomials are characterized in terms of $\D_d^{(0)}$ in \cite[Theorem 3.1]{srs-akiyama}.
Hence, the structure of $\D_d^{(0)}$ is of interest. As was shown in \cite{srs-akiyama}, the case $d = 1$ turns out to be simple since $\D_1^{(0)} = [0,1)$. However, complications arise when $d = 2$. The results in \cite{srs-akiyama-ii} are mainly devoted to the characterization of $\D_2^{(0)}$, particularly which sets are contained in $\D_2^{(0)}$ (refer to Figure 1 of \cite{srs-akiyama-ii} for more details). Moreover, \cite{srs-akiyama} establishes an algorithm to determine when a given point is inside $\D_2^{(0)}$.

We have seen in Section~\ref{MPDS} that there is a relationship between the matrix digit system $(A,\D)$ and $(\sR,X,\sN)$ where $A,\D,\sR,X,\sN$ are as described in the preceding sections. As such, we turn our attention to the characterization of SRS with respect to canonical number systems. We begin by stating the aforementioned results by Akiyama et al.~\cite{srs-akiyama} and Berth\'{e} et al.~\cite{berthe} for the case where $d = 2$.

\begin{theorem} \label{thm:cns-poly}
Let $P(x) = cx^2 + ax + b \in \Z[x]$. Then $P(x)$ is a CNS polynomial if and only if $\pren{\frac{c}{b}, \frac{a}{b}} \in \D_2^{(0)}$.    
\end{theorem}

As such, Theorem \ref{thm:cns-poly} implies that it suffices to look at the rational points of $\D_2^{(0)}$ in order to obtain coefficients for CNS polynomials.

We now use Figure~\ref{fig:dapprox} together with Theorem \ref{thm:cns-poly} to refine the cases we obtained in the previous sections. Like before, let $A = A(\alpha,\beta)$ be the companion matrix of the irreducible polynomial $p(x) = x^2 + \alpha x + \beta \in \Q[x]$ and $c \in \Z^+$ be the unique positive integer such that $q(x) = cp(x) =: cx^2 + ax + b$ is primitive. According to Theorem $\ref{thm:cns-poly}$, $q(x)$ is a CNS polynomial if and only if $\pren{\frac{c}{b},\frac{a}{b}} = \big(\frac{1}{\beta},\frac{\alpha}{\beta}\big) \in \D_2^{(0)}$. In view of Theorem \ref{thm:comm}, a rational point of the form $\big(\frac{1}{\beta},\frac{\alpha}{\beta}\big) \in \D_2^{(0)}$ corresponds to a matrix digit system $(A,\D)$ having zero attractor. Now, let $x = 1/\beta$ and $y = \alpha/\beta$. As a consequence of Theorem \ref{thm:attractor}, we consider the following regions in $\R^2$:
\begin{align*}
    R_1 &= \brc{(x,y) \in \R^2 \mid x > 0 \text{ and } 0 < y \leq 1-x} \\
    R_2 &= \brc{(x,y) \in \R^2 \mid x > 0 \text{ and } x < -y < 1-x} \\
    R_3 &= \brc{(x,y) \in \R^2 \mid x > 0, 0 < -y < 1-x, \text{ and } y \geq -x} \\
    R_4 &= \brc{(x,y) \in \R^2 \mid x < 0 \text{ and } -1-x < -y < 0} \\
    R_5 &= \brc{(x,y) \in \R^2 \mid x < 0 \text{ and } -1-x < y < 0}.
\end{align*}
Then for each $i \in \brc{1,2,3,4,5}$, $(x,y) \in R_i$ if and only if $\alpha$ and $\beta$ satisfy Case $i$ of Theorem \ref{thm:attractor}. 

Observe that the regions $R_1$ and $R_3$, which are regions whose rational points correspond to matrix digit systems with attractor zero, are contained in $\D_2^{(0)}$. On the other hand, the regions $R_2$, $R_4$, and $R_5$ do not intersect with $\D_2^{(0)}$, and they are regions containing rational points with matrix digit systems having a nonzero attractor. Figure \ref{fig:attrregions} illustrates this observation. As such, we can use Figure~\ref{fig:attrregions} to consider cases not covered by Theorem \ref{thm:attractor} and construct the automata corresponding to these cases. Of particular interest are subregions of $\D_2^{(0)}$ that are disjoint from regions $R_1$ and $R_3$. As we shall see later, the number of states needed to construct such automata will increase to accommodate the additional conditions imposed on $\alpha$ and $\beta$.

\begin{figure}[h]
\includegraphics[scale=0.8]{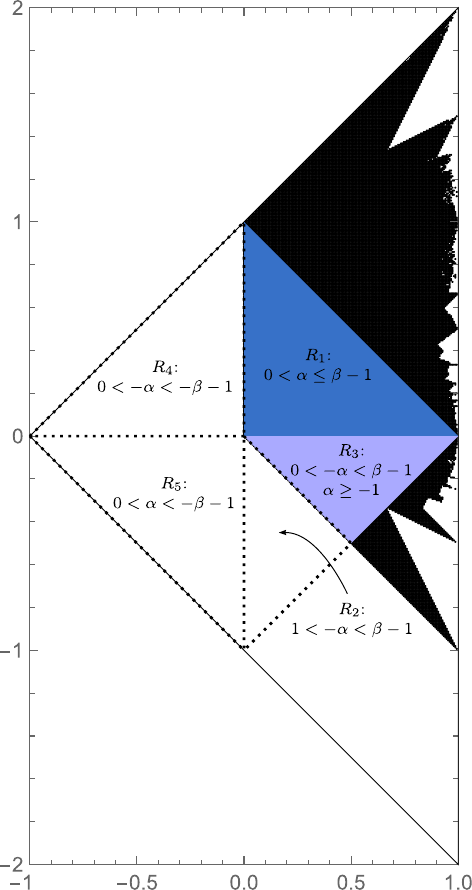}
\caption{Graphs of regions $R_i$ for $i \in \brc{1,2,3,4,5}$ and $\D_2^{(0)}$.}
\label{fig:attrregions}
\end{figure}

It can be shown that whenever $\alpha > 1$ and $\beta - 1 \leq \alpha \leq \beta$, then we obtain exactly the same automata as Case 1 of Theorem \ref{thm:attractor}. We now assume that $0 < \alpha \leq 1$ and $0 < \beta - 1 < \alpha < \beta$. Then under these assumptions, $1 < \beta \leq 2$. We will also impose that $\alpha + \beta > 2$ (which implies that $\alpha>1/2$) and $\beta > 2\alpha$. Apart from the states $\pm P,\dots,\pm T$ introduced in Section~\ref{SMR}, we now introduce new states $\pm E, \dots, \pm H$ defined as follows:
\begin{align*}
    (\sigma^{\pm E})_A &= (\sigma)_A \mp A(2c,0)^\top \mp (2a-c,0)^\top \\
    (\sigma^{\pm F})_A &= (\sigma)_A \pm (2c,0)^\top \\
    (\sigma^{\pm G})_A &= (\sigma)_A \pm A(c,0)^\top \pm (a-2c,0)^\top \\
    (\sigma^{\pm H})_A &= (\sigma)_A \mp A(2c,0)^\top \mp (2a,0)^\top.
\end{align*}
For this case, the relations involving $\sigma^{\pm P}, \sigma^{\pm Q}, \sigma^{\pm R}, \sigma^{\pm S}$ is the same as in Section~\ref{Case1}. The remaining states satisfy the following relations.
\begin{align*}
    (\sigma\nu)^{\pm T} &=
    \begin{cases}
        \sigma^{\mp Q} \, \ol{\nu \pm a \pm c \mp b}, & 0 \leq \nu \pm a \pm c \mp b < b \\
        \sigma^{\pm E} \, \ol{\nu \pm a \pm c \mp 2b}, & 0 \leq \nu \pm a \pm c \mp 2b < b
    \end{cases} \\
    (\sigma\nu)^{\pm E} &=
    \begin{cases}
        \sigma^{\pm G} \, \ol{\nu \mp 2a \pm c \pm b}, & 0 \leq \nu \mp 2a \pm c \pm b < b \\
        \sigma^{\mp F} \, \ol{\nu \mp 2a \pm c}, & 0 \leq \nu \mp 2a \pm c < b
    \end{cases} \\
    (\sigma\nu)^{\pm F} &=
    \begin{cases}
        \sigma^{\pm R} \, \ol{\nu \pm 2c \mp b}, & 0 \leq \nu \pm 2c \mp b < b \\
        \sigma^{\pm H} \, \ol{\nu \pm 2c \mp 2b}, & 0 \leq \nu \pm 2c \mp 2b < b
    \end{cases} \\
    (\sigma\nu)^{\pm G} &=
    \begin{cases}
        \sigma^{\pm T} \, \ol{\nu \pm a \mp 2c \pm b}, & 0 \leq \nu \pm a \mp 2c \pm b < b \\
        \sigma^{\pm S} \, \ol{\nu \pm a \mp 2c}, & 0 \leq \nu \pm a \mp 2c < b 
    \end{cases} \\
    (\sigma\nu)^{\pm H} &=
    \begin{cases}
        \sigma^{\pm G} \, \ol{\nu \mp 2a \pm b}, & 0 \leq \nu \mp 2a \pm b < b \\
        \sigma^{\mp F} \, \ol{\nu \mp 2a}, & 0 \leq \nu \mp 2a < b 
    \end{cases}
\end{align*}

\begin{figure}[h]
\includegraphics[scale=0.9]{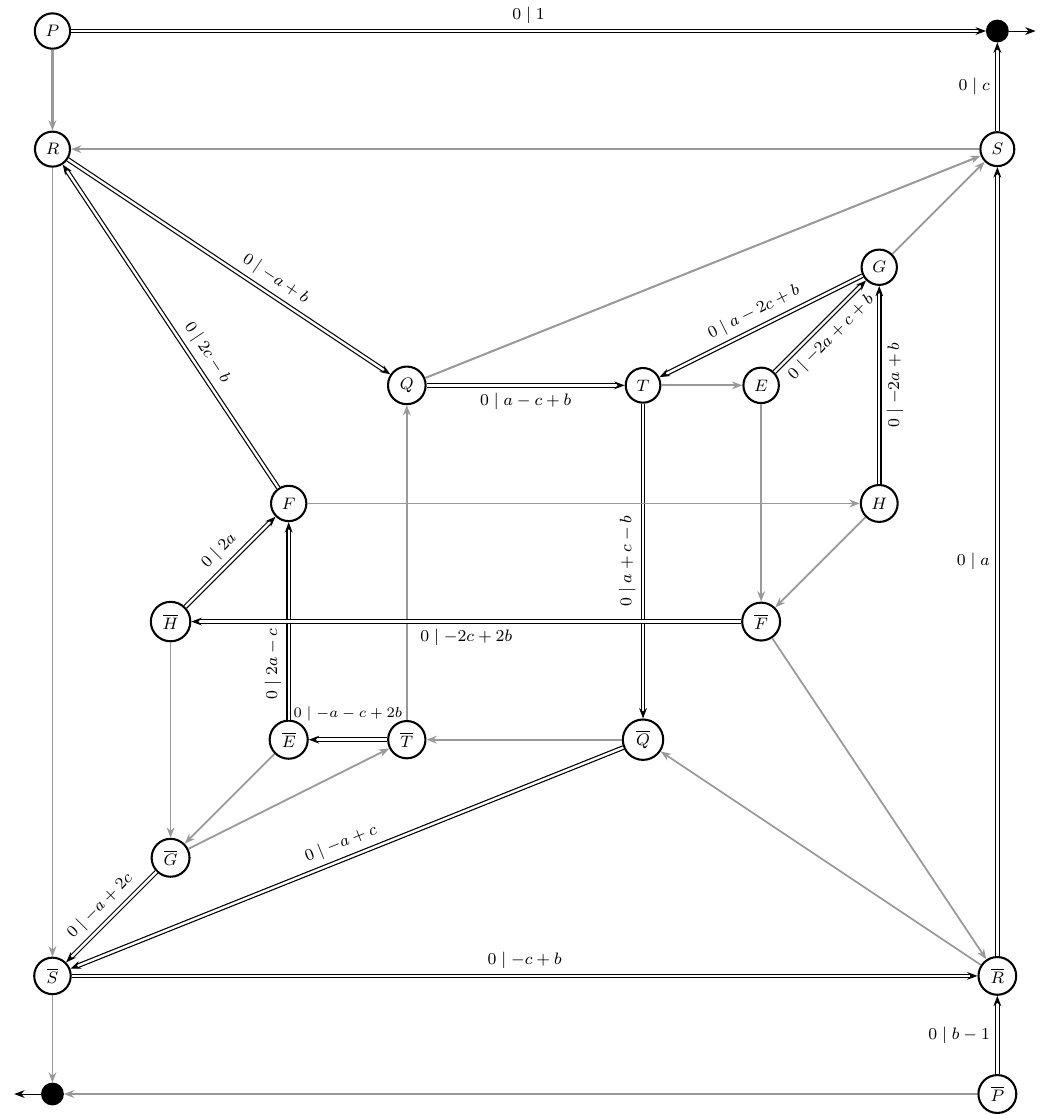}
\caption{Transducer that realizes the addition by $(1,0)^\top$ whenever $0 < \alpha \leq 1$, $0 < \beta - 1 < \alpha < \beta$, $\alpha + \beta > 2$, and $\beta > 2\alpha$.}
\label{fig:automaton51}
\end{figure}

Observe that the condition $\beta > 2\alpha$ is needed for the relations of states $\pm H$ to hold, and $\alpha + \beta > 2$ is required to satisfy those of states $\pm E$ and $\pm G$. The transducer which captures this case is given in Figure~\ref{fig:automaton51}. Here, only edge labels of the form $0 \mid k$ are shown and edges containing a label of this form are displayed as double lines. Observe that any input string of the form $\sigma = 0^\infty w$ with $w$ a finite word over $\D$ will lead to a terminal state once fed to the transducer. Hence, the corresponding output string has the same form as $\sigma$. By the arguments of Cases 1 and 3 of Theorem \ref{thm:attractor}, the dynamical system $\Phi$ has zero attractor. This leads to the following result.

\begin{theorem} \label{thm:automaton51}
Consider the digit system $(A,\D)$ where $A$ is the companion matrix of the irreducible polynomial $p(x) = x^2 + \alpha x + \beta
\in \Q[x]$ and $\D = \brc{(0,0)^\top,(1,0)^\top,\dots,(\abs{b}-1,0)^\top}$. Let $\A_\Phi$ be the attractor of the dynamical system $\Phi$ induced by $(A,\D)$, $c$ be the unique positive integer for which the polynomial $q(x):=cp(x)$ is a polynomial with coprime integer coefficients, and $a = c\alpha, b = c\beta$. Then the following hold:
\begin{enumerate}
    \item If $\alpha > 1$ and $\beta - 1 \leq \alpha \leq \beta$, then $\A_\Phi = \brc{0}$.
    \item If $0 < \alpha \leq 1$ and $0<\beta - 1 < \alpha < \beta$, then $\A_\Phi = \brc{0}$ provided that one of the following conditions hold.
    \begin{enumerate}
        \item $\alpha + \beta > 2$ and $\beta > 2\alpha$
        \item $\alpha + 2 < 2\beta$ and $\beta \leq 2\alpha$
    \end{enumerate}
\end{enumerate}
\end{theorem}

As in the preceding discussion, let $x = 1/\beta$ and $y = \alpha/\beta$ and consider the following regions in $\R^2$:
\begin{align*}
    S_1 &= \brc{(x,y) \in \R^2 \mid y > x \text{ and } 1-x \leq y \leq 1} \\
    S_2 &= \brc{(x,y) \in \R^2 \mid 0 < y \leq x, 0<1-x < y < 1, y+1 > 2x, \text{ and } 2y < 1} \\
    S_3 &= \brc{(x,y) \in \R^2 \mid 0 < y \leq x, 0<1-x < y < 1, y+2x < 2, \text{ and } 2y \geq 1}.
\end{align*}
Then $(x,y) \in S_1$ (respectively $S_2$, $S_3$) precisely when $\alpha$ and $\beta$ satisfy the conditions of Case 1 (respectively Case 2(a), Case 2(b)) in Theorem~\ref{thm:automaton51}. The graphs of regions $R_1$, $R_3$, $S_1$, $S_2$, and $S_3$ are shown in Figure~\ref{fig:regions}. Notice that all their interiors are contained in $\D_2^{(0)}$.

\begin{figure}[h]
\centering

\includegraphics[scale=0.7]{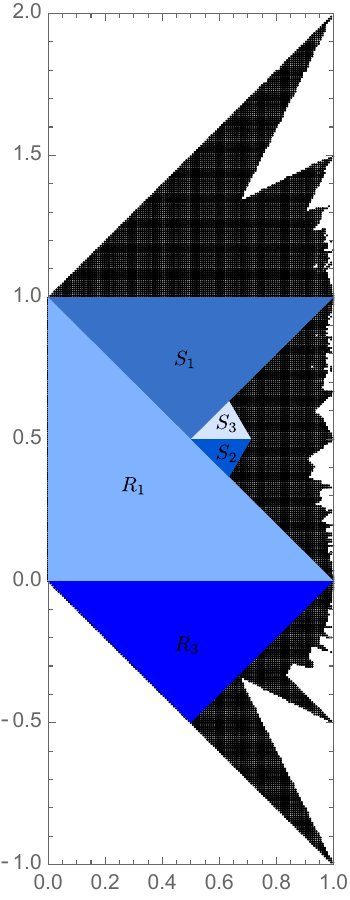}
\caption{Graphs of regions $R_1$, $R_3$, $S_1$, $S_2$, and $S_3$ inside $\D_2^{(0)}$. Every rational point in these regions corresponds to a matrix digit system with zero attractor.}
\label{fig:regions}
\end{figure}

It remains to establish Case 2(b) of Theorem \ref{thm:automaton51}. The conditions of Case 2 still hold, but we now assume that $\beta \leq 2\alpha$. As a result, the relations for states $\pm H$ no longer hold, and the transducer in Figure~\ref{fig:automaton51} is not anymore applicable. Given this, we now define the digit strings $\sigma^{\pm I}$ and $\sigma^{\pm J}$ as follows:
\begin{align*}
    (\sigma^{\pm I})_A &= (\sigma)_A \pm A(2c,0)^\top \pm (2a-2c)^\top \\
    (\sigma^{\pm J})_A &= (\sigma)_A \pm A(c,0)^\top \pm (a+2c,0)^\top.
\end{align*}

Furthermore, we also impose that $\alpha + 2 < 2\beta$. Together with the condition $\beta \leq 2\alpha$, this implies that $\alpha + \beta > 2$. Hence all previous relations hold except those for states $\pm H$. The following are the relations for states $\pm H, \pm I, \pm J$.
\begin{align*}
    (\sigma\nu)^{\pm H} &=
    \begin{cases}
        \sigma^{\pm I} \, \ol{\nu \mp 2a \pm 2b}, & 0 \leq \nu \mp 2a \pm 2b < b \\
        \sigma^{\pm G} \, \ol{\nu \mp 2a \pm b}, & 0 \leq \nu \mp 2a \pm b < b
    \end{cases} \\
    (\sigma\nu)^{\pm I} &=
    \begin{cases}
        \sigma^{\pm J} \, \ol{\nu \pm 2a \mp 2c \pm b}, & 0 \leq \nu \pm 2a \mp 2c \pm b < b \\
        \sigma^{\pm F} \, \ol{\nu \pm 2a \mp 2c}, & 0 \leq \nu \pm 2a \mp 2c < b
    \end{cases} \\
    (\sigma\nu)^{\pm J} &=
    \begin{cases}
        \sigma^{\mp Q} \, \ol{\nu \pm a \pm 2c \mp b}, & 0 \leq \nu \pm a \pm 2c \mp b < b \\
        \sigma^{\pm E} \, \ol{\nu \pm a \pm 2c \mp 2b}, & 0 \leq \nu \pm a \pm 2c \mp 2b < b
    \end{cases}
\end{align*}

The condition $\alpha + 2 < 2\beta$ is required for the relations of states $\pm I$ and $\pm J$ to hold, and $2\alpha \geq \beta$ for the states $\pm H$. The transducer that realizes addition by $(1,0)^\top$ in this case is given in Figure~\ref{fig:automaton52}. Only edge labels of the form $0 \mid k$ are displayed, and edges with such labels
are drawn as double lines. Observe that two double lines originate from state $I$, since the transition depends on whether $a=c$ or not.  As with the previous case, feeding an infinite string of the form $\sigma = 0^\infty w$ with $w$ a finite word over $\D$ into the automaton produces a path ending at a terminal state. Thus the induced output string also has the same form as $\sigma$ and hence the corresponding digit system has zero attractor. This completes the proof.

\begin{figure}[h]
\includegraphics[scale=0.9]{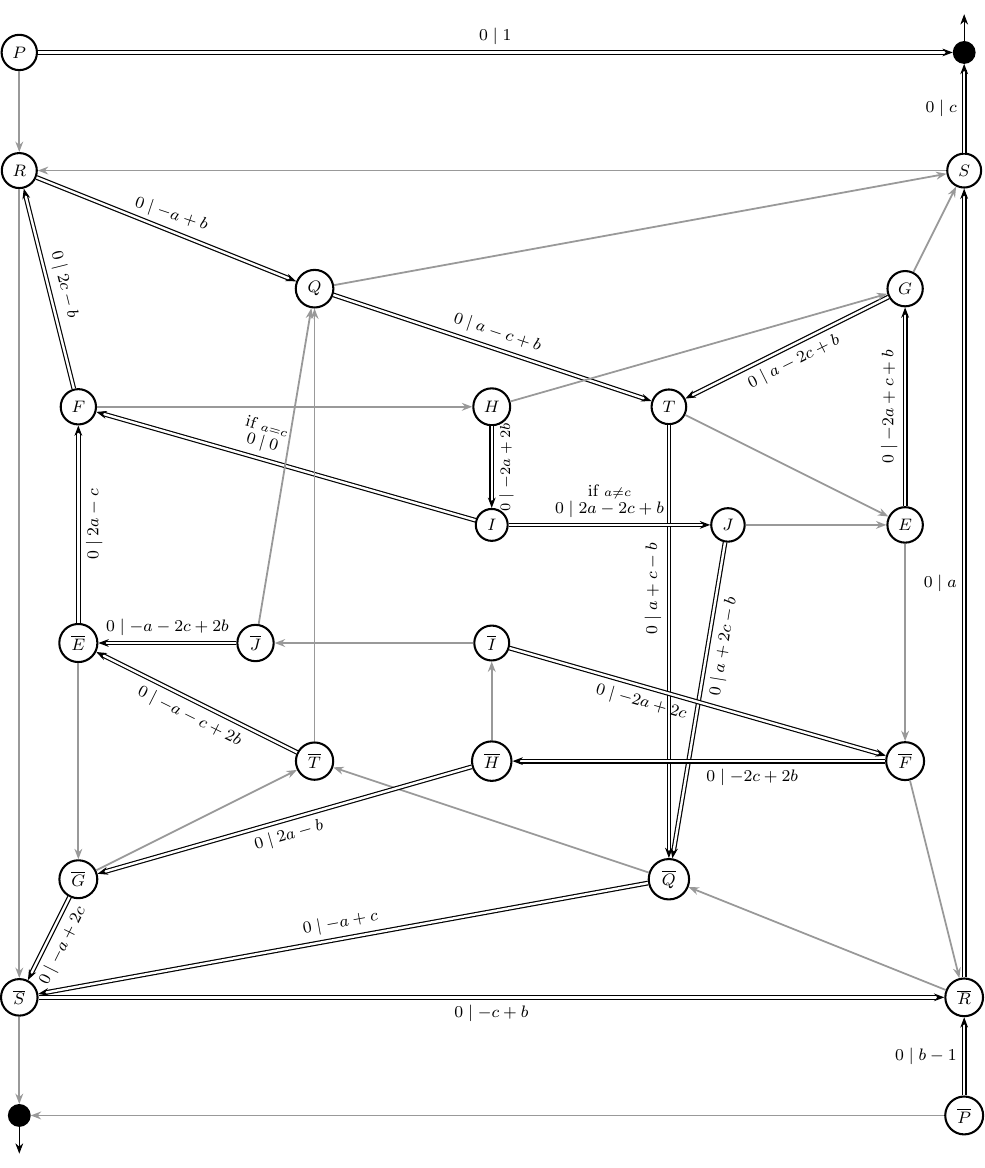}
\caption{Transducer that realizes the addition by $(1,0)^\top$ whenever $0 < \alpha \leq 1$, $0 < \beta - 1 < \alpha < \beta$, $\alpha + 2 < 2\beta$, and $\beta \leq 2\alpha$.}
\label{fig:automaton52}
\end{figure}

\end{document}